\documentclass{amsart}
\usepackage{bbm}
\usepackage{amsmath,amsthm,amsfonts,amssymb,amscd,mathrsfs}
\usepackage{psfrag,graphicx}

\usepackage{epic,eepic}
\usepackage{color}
\usepackage{ebezier}

\theoremstyle{plain}

\newtheorem{Thm}{Theorem}[section]
\newtheorem{Lem}[Thm]{Lemma}
\newtheorem{Prop}[Thm]{Proposition}
\newtheorem{Cor}[Thm]{Corollary}

\theoremstyle{remark}

\long\def\begcom#1\endcom{}

\newcommand{\length}{\operatorname{\length}}

\newcommand{\Diff}{\operatorname{Diff}}

\def\n{\noindent}

\def\Diff{\operatorname{Diff}}

\def\supp{\operatorname{supp}}

\def\length{\operatorname{length}}

\def\vep{\varepsilon}

\begin{document}

\title[Upper bounds on measure theoretic tail entropy]
      {Upper bounds on measure theoretic tail entropy for dominated splittings}

\author[Y. Cao, G. Liao and Z. You]
{Yongluo Cao$^1$,  Gang Liao$^2$ and  Zhiyuan You$^3$}

\email{ylcao@suda.edu.cn}

\email{lg@suda.edu.cn}

\email{suzhouyou@qq.com}



\thanks{$^{1}$ Department of Mathematics,  East China Normal University,
	Shanghai 200062, China \&  School of Mathematical Sciences, Center for Dynamical Systems and Differential Equations,  Soochow University,
	Suzhou 215006,  China; $^{2,3}$ School of Mathematical Sciences, Center for Dynamical Systems and Differential Equations,  Soochow University,
	Suzhou 215006,  China. Yongluo Cao was partially supported by NSFC (11771317,  11790274),  Science and Technology Commission
	of Shanghai Municipality (18dz22710000); $^{2}$  Corresponding author. Gang Liao was partially supported by NSFC (11701402,   11790274), BK 20170327 and Jiangsu province ``Double Plan"}


\maketitle

\begin{abstract}
 For differentiable dynamical systems with dominated splittings, we give upper estimates on the measure-theoretic tail entropy in terms of Lyapunov exponents. As our primary application, we verify the upper semi-continuity of metric entropy in various settings with domination.
\end{abstract}

\section{Introduction.}

Let $f$ be a homeomorphism on a compact metric space $M$.   For $K\subset M$, $n\in \mathbb{N}$ and any observable scale $\vep>0$, a subset $K_1 \subset K $ is
called $(n, \vep)$-spanning for $K$ if for any $x\in
K$ there exists $y\in K_1$ such that $ d(f^i(x),f^i(y))\leq
\vep$, $\forall \,i \in[0, n)$.  Let
$r_n(f,K,\varepsilon)$ denote the smallest cardinality of any
$(n,\vep)$-spanning set of $K$. The $\vep$-topological entropy
of $K$ is defined by
$$h(f,K,\varepsilon)=\limsup_{n\rightarrow\infty}\frac{1}{n}\log
r_n(f,K,\varepsilon).$$ 
The topological entropy of $f$ on $K$ is defined by
$$h(f,K)=\lim_{\varepsilon\rightarrow0}h(f,K,\varepsilon).$$
  For $x \in M$, $n\in \mathbb{N}$ and $\vep
> 0 $, let
\begin{eqnarray*}
	B_n(f, x,  \vep)&=& \{ y \in M: d(f^i(x), f^i(y))<\vep,\, |i|<n\},\\[2mm]
	 B_{\infty}(f, x, \vep)&=&\cap_{n\in \mathbb{N}}\, B_n(f, x,  \vep),
\end{eqnarray*}
then the $\vep$-tail entropy at $x$ is defined by  
$$h^{*}(f, x, \varepsilon)=h(f,\, B_{\infty}(f, x, \varepsilon)).$$
Tail entropy has  been studied broadly since the pioneering works of Bowen \cite{Bowen} and Misiurewicz \cite{Mis2} in 1970's, because of its fundamental role in the estimates of entropy in both the  topological and measure theoretic sense.  

Given a compact  $f$-invariant set $\Lambda\subset M$ and $\vep>0$, denote $$h^*(f, \Lambda, \vep)=\sup_{x\in \Lambda}h^{*}(f, x, \vep).$$  We say $f$ on  $\Lambda$ is entropy expansive  \cite{Bowen}   if there exists
$\delta>0$ such that
$h^*(f, \Lambda, \delta)=0$
and asymptotically  entropy expansive 
\cite{Mis2} if
$\lim_{\delta\rightarrow0}h^*(f, \Lambda, \delta)=0.$ 
 Both of these properties imply the upper semi-continuity of metric entropy. 

Denote by $\mathcal{M}_{inv}(f, \Lambda)$ and $\mathcal{M}_{erg}(f, \Lambda)$ the sets of $f$-invariant and ergodic $f$-invariant Borel probability measures  on a compact  $f$-invariant  set $\Lambda\subset M$, respectively.    Consider $\mu\in \mathcal{M}_{erg}(f, \Lambda)$,  then $h^*(f,x, \vep)$ is a constant  for $\mu$-a.e., $x$ (Proposition 2.8 of \cite{CY}), which we denote by  $h^*(f,\mu, \vep)$. In general, when $\mu\in \mathcal{M}_{inv}(f, \Lambda) $, denoting its ergodic decomposition  $\mu=\int_{\mathcal{M}_{erg}(f, \Lambda)}d\tau(m)$, we define the measure theoretic tail entropy of $\mu$ as 
$$h^{*}(f,\mu,\vep)=\int_{\mathcal{M}_{erg}(f, \Lambda)}h^*(f,m, \vep)d\tau(m).$$
 By the tail variational principle \cite{Dows, Burguet}, one has
$$\lim_{\vep\rightarrow0} \,\,\sup_{\mu\in \mathcal{M}_{inv}(f, \Lambda)}h^*(f, \mu,\vep)
=\lim_{\vep\rightarrow0} h^*(f,\Lambda, \vep). $$
However, it is unknown if $ \sup_{\mu\in \mathcal{M}_{inv}(f, \Lambda)}h^*(f, \mu,\vep)
= h^*(f,\Lambda, \vep) $ for any $\vep>0$. 

Tail entropy measures the local dynamical complexity in the process of observation with respect to  the evolutions of dynamical systems.  It is known that uniformly hyperbolic systems are  entropy expansive, and so are all diffeomorphisms away from tangencies \cite{LVY}.  As a more general concept, dominated splitting   exhibiting  uniformly hyperbolic  behavior  on projective bundles,  is admitted by plenty of systems beyond uniform hyperbolic systems \cite{Shub, Mane, BonattiV, BochiV, ACW}.  In the present paper, we attempt to study tail entropy in the setting of dominated splitting.

Let  $\Diff(M)$ be the space of $C^1$ diffeomorphisms on a compact boundaryless Riemannian manifold $M$.   For $f\in \Diff(M)$,   a splitting
$T_{\Lambda}M=E_1\oplus_{<} \cdots\oplus_{<} E_{\ell}$ over a compact  $f$-invariant set $\Lambda\subset M$ is said to be  dominated if there exists $L\in \mathbb{N}$ such that for any $x\in \Lambda$,  $v\in
E_i(x)$,  $w\in E_j(x)$ with $\|v\|=\|w\|=1$ and  $1\le i<j\le \ell$, 
$$\|D_{x}f^L v\|\leq \frac12\|D_{x}f^L w\|.$$
Taking an adapt metric \cite{Nik}, we may assume $L=1$ in the following discussions for dominated splittings.

Recall that the  geometric divergent rate of any $x\in M$ relative to a direction $v\in T_xM$ is given  by the limit
\begin{eqnarray*}\label{limit}\lim_{n\to \infty}\frac{1}{n}\log \|D_xf^nv\|,\end{eqnarray*}
which exists and is called the Lyapunov exponent along $v$, for almost every point $x$ of every  $f$-invariant measure by Oseledets theorem \cite{Oseledec}. For a dominated splitting $T_{\Lambda}M=E_1\oplus_{<} \cdots\oplus_{<} E_{\ell}$ over $\Lambda$, for the purpose of studying  the approximation process of Lyapunov exponents with respect to the evolution time $N$,  we define for any $1\le i\le \ell$, 
\begin{eqnarray*}
\Delta_{f}^{\pm}( x, E_i; N)&=&\lim_{n\rightarrow \pm\infty}\frac{1}{|nN|}\sum_{k=0}^{n-1}\log^+ \|(D_{f^{kN}(x)}f^{\pm N}\mid_{E_i})^{\wedge}\|,\\[2mm]
\Delta_{f}( x, E_i;  N)&=& \min \big{\{} \Delta_{f}^{+}( x, E_i,  N),\,  \Delta_{f}^{-}( x, E_i,  N) \big{\}},
 \\[2mm]\Delta_{f}( x;  N)&=& \min \big{\{} \Delta_{f}( x, E_i,  N):\,  1\le i\le \ell \big{\}}.
\end{eqnarray*}
where $\log^+ t=\max\{0, \log t\}$, and for a linear transformation  $T: X_1 \to X_2$ between  two finite dimensional linear spaces $X_1$ and $X_2$,   $T^{\wedge}$  denotes the map on the exterior algebra
of  $X_1$ (In this manner,  $\|T^{\wedge}\|$ is the maximum  of the absolute values of Jacobians of $T$ on any linear subspace of $X_1$). 
Denote   
$\Delta_{f}^+(x, E_i)$ ($\Delta_{f}^-(x, E_i)$) as the sum of  positive Lyapunov exponents on $E_{i}$ (the sum of the absolute values of negative Lyapunov exponents on $E_{i}$),  then by Oseledets theorem \cite{Oseledec} one could get that for $\mu$-a.e., $x$ of every 
$\mu\in \mathcal{M}_{inv}(f, \Lambda) $, 
\begin{eqnarray*}
 \Delta_{f}(x, E_i; N) &\rightarrow&  \Delta_{f}(x, E_i   ) :=\min \big{\{} \Delta_{f}^{+}(x, E_i),\,  \Delta_{f}^{-}( x, E_i) \big{\}} \quad \text{as}\,\,N\to +\infty,\\[2mm]
 \Delta_{f}(x; N) &\rightarrow&  \Delta_{f}(x):=\min \big{\{} \Delta_{f}(x, E_i):\,  1\le i\le \ell \big{\}} \quad \text{as}\,\,N\to +\infty.
\end{eqnarray*}
 For $\mu\in \mathcal{M}_{inv}(f, \Lambda) $,  let 
\begin{eqnarray*}
	\Delta_{f}(\mu, E_i)&=&\int \Delta_{f}(x, E_i) d\mu(x),\\[2mm]
\Delta_{f}(\mu; N)&=&\int \Delta_{f}(x; N) d\mu(x),\\[2mm]
\Delta_{f}(\mu)&=&\int \Delta_{f}(x) d\mu(x).
\end{eqnarray*}

By analyzing the approximation process of Lyapunov exponents,   we can get the estimates concerning the relationship between the scale of measure theoretic tail entropy and the evolution time.  

\begin{Thm}\label{uppersemi}
	
	Let  $f\in \Diff(M)$ and   
	$T_{\Lambda}M=E_{1} \oplus_{<}  \cdots\oplus_{<} E_{\ell}$ be a dominated splitting  over a compact  $f$-invariant set $\Lambda$,
	then there exists a sequence $\{\vep_N\}_{N\in \mathbb N}$ of positive numbers with $\lim_{N\rightarrow +\infty} \vep_N=0$ such that 
	$$\lim_{N\to +\infty}\, \sup_{\mu\in \mathcal{M}_{inv}(f, \Lambda)} \Big{(}h^*(f,\mu, \vep_N)-\Delta_f(\mu; N)\Big{)} \leq 0.$$ 	
	In particular, it holds that 
	$$\lim_{\vep\to 0} h^*(f,\mu, \vep) \leq \Delta_f(\mu)$$
for any $\mu\in \mathcal{M}_{inv}(f, \Lambda)$.

\end{Thm}

\noindent{\it Remark.} The tail entropy  was studied with respect to a dominated splitting over the manifold $M$ by Buzzi, Crovisier and Fisher (Theorem 7.6 of \cite{BCF}). In Theorem \ref{uppersemi}, we focus on the uniform difference between the  tail entropy of measures and the Lyapunov exponents of themselves relative to the evolution time.

In order to use the measure theoretic tail entropy to estimate  the difference between the full metric entropy $h_{\mu}(f)$ and the metric entropy  $h_{\mu}(f, \mathcal P)$ with respect to some partition $\mathcal P$, we further establish the following theorem which is a strengthening version of Proposition 2.1 of \cite{LSW} for  the use of infinite Bowen balls in the definition of tail entropy here.   

\begin{Thm}\label{tail of metric entropy} Let $M$ be a compact metric space  and $f: M\to M$  a
	homeomorphism with finite topological entropy. For any    $\mu \in \mathcal{M}_{inv}(f,M)$,  it holds that  
	$$h_{\mu}(f)-h_{\mu}(f, \mathcal P)\leq h^*(f,\mu, \rho)$$
	\n for any finite measurable partition $\mathcal P$ with $diam(\mathcal P) \leq \rho.$

\end{Thm}

In what follows, applying Theorems \ref{uppersemi} and \ref{tail of metric entropy}, we may deduce the upper semi-continuity property of metric entropy in case that $\Delta_f(\mu)=0$.

\begin{Cor}\label{uppersemicon}
	
	Let  $f\in \Diff(M)$ and   
	$T_{\Lambda}M=E_{1} \oplus_{<}  \cdots\oplus_{<} E_{\ell}$ be a dominated splitting  over a compact  $f$-invariant set $\Lambda$.    
Then  the metric entropy map in   $\mathcal M_{inv}(f, \Lambda)$  is upper semi-continuous at any $\mu$ with $\Delta_f(\mu)=0$.
	
\end{Cor}

In fact, Corollary \ref{uppersemicon} could  give rise to the upper semi-continuity  of metric entropy for plenty of systems with domination.

Combining with \cite{mane}  and Theorem 3.3 of \cite{ABC},  for a $C^1$ generic $f\in \Diff(M)$, a generic element $\mu$ in  $ \mathcal M_{erg}(f, M)$  admits dominated Oseledets splittings, so the corresponding $\Delta_f(\mu)=0$, which implies, by Corollary \ref{uppersemicon},  the upper semi-continuity of metric entropy at $\mu$ in  $\mathcal M_{inv}(f, \supp(\mu))$, where $\supp(\mu)$ is the support of $\mu$.    Moreover, given a homoclinic class $H$,  if we denote  by $\mathcal M_{per}(H)$ the closure of convex hull of periodic measures supported on $H$,  then  by Theorem 3.1' of \cite{ABC}, $\supp(\mu)=H$ and $h_{\mu}(f)=0$ for generic $\mu\in \mathcal M_{per}(H) $, thus  we can obtain 

\begin{Cor}\label{uppersemigeneric}
	
 For a $C^1$ generic $f$ in $\Diff(M)$ and any homoclinic class $H$ of  $f$, 	the set of continuity points of metric entropy in $\mathcal M_{inv}(f, H)$ includes a residual subset of  $ \mathcal M_{per}(H)$.
	
\end{Cor}

 In the setting of conservative systems, for a $C^1$ generic $f$ in $\Diff_{vol}(M)$ which denotes the space of  $C^1$ diffeomorphisms on $M$ preserving the volume measure $vol$,  by \cite{BochiV, ACW},  the Oseledets splitting of  $vol$ is dominated. Thus,
 
 \begin{Cor}\label{uppersemileb}
 	
 For a $C^1$ generic $f$ in $\Diff_{vol}(M)$,   
 	the volume measure $vol$ is an upper semi-continuity point of metric entropy in $ \mathcal M_{inv}(f, M)$.
 	
 \end{Cor}

Besides, by Corollary \ref{uppersemicon}, we may also get an alternative criteria for the upper semi-continuity  of metric entropy for dominated splittings consisting of bundles  without mixed behavior or
of one dimension in \cite{LVY, CY, ZYC}, since  $\Delta_f(\mu)=0$ is satisfied in those contexts.

\section{Dynamics of foliations}
Let $f\in \Diff(M)$,   $\Lambda$ be a compact  $f$-invariant set and there exists a   dominated splitting 
$T_{\Lambda}M=E_{1} \oplus_{<}  \cdots\oplus_{<} E_{\ell}$ over $\Lambda$. For $1\leq i\le j\le \ell$, denote $E_{i(i+1)\cdots j}=E_i\oplus \cdots \oplus E_j$.  Let $\xi_0$ be a positive lower bound for the angles between any pair of  bundles $E_{i}$ and $E_{j}$, $1\le i\neq j\le \ell$.  By \cite{HPS, BW},  with respect to the given domination structure, one may have a family of local invariant fake foliations. In the following content, given a foliation $\mathcal F $ and a point $y$ in
the domain, we denote by $\mathcal F(y)$ the leaf through $y$  and by  $\mathcal F(y, \rho)$ the neighborhood of  radius $\rho$ around $y$  inside the leaf. 

\begin{Prop} For any $\xi\in (0, \xi_0/4)$, there exist $0<\rho_2<\rho_1$ such that the neighborhood $B(x, \rho_1)$ of every  $x\in \Lambda$ admits   foliations $\{\mathcal F^*_x:\,x\in \Lambda\}$, $*\in \{i(i+1)\cdots j: 1\leq i\le j\le \ell\}$, such that for any $y\in B(x, \rho_1)$ and $*\in \{i(i+1)\cdots j: 1\leq i\le j\le \ell\}$,  
	\begin{itemize}\item[(i)] almost tangency: $T_y \mathcal F^*_x(y)$ lies in a cone of width $\xi$ of $E_*(x)$;
		\\
		\item[(ii)] local invariance: $f^{\pm}\mathcal F^*_x(y, \rho_2) \subset \mathcal F^*_{f^{\pm}(x)}(f^{\pm}(y))$;\\
		\item[(iii)] coherence: $\mathcal F^*_x$ is subfoliated by  $\mathcal F^{\#}_x$ whenever $\#$ is a subsentence of $*$.
	\end{itemize}
	
\end{Prop}

Along the leaves of foliations $\mathcal F_x^*$, we could  define the projections  as follows: for
 $y\in B(x,\rho_1)$,  $1\leq i\leq \ell-1$, let 
\begin{eqnarray*}
[y]^{ 1\cdots i}_x&=&\mathcal F_x^{(i+1)\cdots \ell}(y) \cap\mathcal F_x^{1\cdots i}(x),\\[2mm]
[y]^{ (i+1)\cdots \ell }_x&=&\mathcal F_x^{1\cdots i}(y) \cap\mathcal F_x^{(i+1)\cdots \ell}(x),
\end{eqnarray*}
wherever they are well defined. 
The almost tangency property (i) and the uniform positive lower bound among  angles of different bundles $ E_*$  make us be able to choose a constant  $C_1>0$  such that for any $\rho\in (0,\rho_1/C_1)$, $y\in B(x, \rho)$ and  $*\in \{1\cdots i,\, (i+1)\cdots \ell: 1\le i\le \ell-1\}$,  one has 
$$[y]^{*}_x\in \mathcal F_x^*(x, C_1\rho).   $$
  By taking some local trivialization of
 the tangent bundle, for any $N\in \mathbb N$ and  $\rho\in(0,\rho_1/C_1)$, 
 we define \begin{eqnarray*}
 	\sigma(N,\rho)&=&\max\Big{\{}\log\big{(}\frac{\|(D_{x_1}f^{\pm N} )^{\wedge_k}\|}{\|(D_{x_2}f^{\pm N})^{\wedge_k} \|}\big{)}:\,x_j\in \mathcal F^*_x(x, C_1\rho),\,\,j=1,2,\\[2mm]
 	&&1\le k\le \dim E_*(x),\, *\in \{1\cdots i,\, (i+1)\cdots \ell: 1\le i\le \ell-1\},\, x\in \Lambda  \, \Big{\}}. 
 \end{eqnarray*}
Denote $e^{P}=\max\{\|D_x^{\pm}f\|: x\in M\}$.  For any $N\in \mathbb{N}$, one may let $\xi$ and $\rho_1$ small such that  $\rho(N)=\rho_1e^{-NP}/C_1$ satisfying $\sigma(N,\rho(N))<1/N$. 

\begin{Lem}[Pliss\cite{Pliss}]\label{lempliss}
	Let $b_{0}\leq c_{2}<c_{1}$ and
	$\theta=\dfrac{c_{1}-c_{2}}{c_{1}-b_{0}}.$  Given real numbers
	$b_{1},\cdots, b_{T}$ with $\sum\limits_{i=1}^{T}b_{i}\leq c_{2}T $
	and $b_{i}\geq b_{0}$ for every $i$,  there exists $\tau\geq
	T\theta$ and $1\leq k_{1}<k_{2}<\cdots < k_{\tau}\leq T $, such that
	$$\sum\limits_{i=k+1}^{k_{j}}b_i\leq c_{1}(k_{j}-k) , \quad\quad0\leq k<k_{j}, \ \ 1\leq j\leq \tau.$$
\end{Lem}

\begin{Lem}\label{local ball and foliation}
There exists $N_0>0$ such that for any $N\ge N_0$ and $\mu\in \mathcal M_{inv}(f, \Lambda)$, for $\mu$-a.e.,  $x$,  $B_{\infty}(f,x, \rho(N))=\{x\}$ or $B_{\infty}(f,x, \rho(N))\subset \mathcal F_{x}^i(x,C_1 \rho(N))$ for some $i\in \{1, \cdots, \ell\}$.

\end{Lem}

\begin{proof}
Let $N_0= [\frac{16}{\log2}]+1$. So,    $\sigma(N,\rho(N))<\frac{\log2}{16}$ for any $N\ge N_0$. For $\mu\in \mathcal M_{inv}(f, \Lambda)$ and  $1\le i\le \ell$, we denote for $\mu$-a.e., $x$, 
\begin{eqnarray*}a_{E_i}^+(x)&=&\lim_{n\rightarrow \infty} \frac{1}{|nN|}\sum_{k=0}^{n-1} \log m(D_{f^{k{N}}(x)}f^{N}\mid_{E_i}),\\[2mm]
a_{E_i}^-(x)&=&\lim_{n\rightarrow \infty} \frac{1}{|nN|}\sum_{k=0}^{n-1} \log \|D_{f^{-k{N}}(x)}f^{-{N}}\mid_{E_i}\|.	
\end{eqnarray*}
Let $$i_0(x)=\min\{1\le i\le \ell:\,  a_{E_i}^+(x)>\frac{\log2}{2}\},$$
then by the domination $T_{\Lambda}M=E_1\oplus_{<}\cdots \oplus 
_{<}E_{i_0(x)-1}\oplus_{<} E_{i_0} \oplus_{<}\cdots \oplus_{<} E_{\ell}$, one has  \begin{equation*}\begin{cases} &a_{E_{i_0(x)-1}}^-(x)\le\frac{\log2}{2}; \\[2mm]
 &a_{E_{j}}^-(x)\le-\frac{\log2}{2},\,\,\,\forall\,j\in [1,\, i_0(x)-2];\\[2mm]
	 &a_{E_{j}}^+(x)\ge \frac{\log2}{2},\,\,\,\forall\,j\in [i_0(x),\, \ell].
 \end{cases}
\end{equation*}
Hence  there exist $1<n_1<n_2<\cdots<n_t<\cdots$ such that for any $t\in \mathbb{N}$, 
$$\frac{1}{n_t}\sum_{k=0}^{n_t-1} \log m(D_{f^{k{N}}(x)}f^{{N}}\mid_{E_{i_0(x)}\oplus\cdots\oplus E_{\ell}})>\frac{\log2}{4},$$
i.e., 
$$\frac{1}{n_t}\sum_{k=1}^{n_t} \log \|D_{f^{k{N}}(x)}f^{-{N}}\mid_{E_{i_0(x)}\oplus\cdots\oplus E_{\ell}}\|<-\frac{\log2}{4}.$$
 Let $b_0=-NP, c_1=-\frac{\log2}{4}, c_2=-\frac{\log2}{8}$  and $\theta=\dfrac{c_{1}-c_{2}}{c_{1}-b_{0}}$. Applying Lemma  \ref{lempliss}, for each $n_t$, we can find $\tilde{n}_t\in [\theta n_t, n_t]$ such that

$$\sum\limits_{k=j+1}^{\tilde{n}_t}\log \|D_{f^{k{N}}(x)}f^{-{N}}\mid_{E_{i_0(x)}\oplus\cdots\oplus E_{\ell}}\|\leq -\frac{\log2}{8}(\tilde{n}_{t}-j) , \quad\quad\,\, 0\leq j<\tilde{n}_t.$$
By the choice of $N\ge N_0$,  it holds that
\begin{eqnarray*}
&&\|Df^{-{N}}_{z}\mid_{T_z\mathcal F^{i_0(x)(i_0(x)+1)\cdots\ell}_{y}(z)}\|\\[2mm]&\le&  2^{\frac{1}{16}}\|Df^{-{N}}_{y}\mid_{E_{i_0}\oplus\cdots\oplus E_{\ell}}\|,\quad\quad \forall z\in \mathcal F^{i_0(i_0+1)\cdots\ell}_{y}(y, C_1\rho(N)),\,\,\, \forall\,y\in \Lambda.
\end{eqnarray*}
Therefore, for $ 1\leq j\le \tilde{n}_t$, $$ f^{-j{N_0}}(\mathcal F^{i_0(i_0+1)\cdots\ell}_{f^{\tilde{n}_t}(x)}(f^{\tilde{n}_t}(x), C_1\rho(N)))\subset \mathcal F^{i_0(x)(i_0(x)+1)\cdots\ell}_{f^{\tilde{n}_t-j}(x)}(f^{\tilde{n}_t-j}(x), 2^{-\frac{(\tilde{n}_t-j)}{16}}C_1\rho(N)). $$
For any $y\in B_{\infty}(f, x, \rho(N))$,  $$[f^n(y)]^{i_0(x)(i_0(x)+1)\cdots \ell}\in \mathcal F^{i_0(x)(i_0(x)+1)\cdots\ell}_{f^{n}(x)}(f^{n}(x), C_1\rho(N)).$$ By the local invariance of fake foliations, 
$$[y]^{i_0(x)(i_0(x)+1)\cdots \ell}= f^{-n}([f^n(y)]^{i_0(i_0+1)\cdots \ell}),\quad \forall\,n\in \mathbb{N}.$$
Specially,  $$[y]^{i_0(x)(i_0(x)+1)\cdots \ell}= f^{-\tilde{n}_t}([f^{\tilde{n}_t}(y)]^{i_0(x)(i_0(x)+1)\cdots \ell})\in  \mathcal F^{i_0(x)(i_0(x)+1)\cdots\ell}_{x}(x,  2^{-\frac{\tilde{n}_t}{16}}C_1\rho(N)). $$ 
Letting $t\to +\infty,$ we get that $$[y]^{i_0(x)(i_0(x)+1)\cdots \ell}=\{x\}.$$
Similarly, one can deduce $[y]^{1\cdots (i_0(x)-2)}=\{x\}$ since $a_{E_{j}}^-(x)\le-\frac{\log2}{2}$ for any $j\in [1,\, i_0(x)-2]$. Then 
\begin{eqnarray*}
B_{\infty}(f,x, \rho(N))&\subset& \mathcal F_{x}^{1\cdots (i_0(x)-1)}(x,C_1\rho(N) ) \cap \mathcal F_{x}^{(i_0(x)-1)\cdots\ell}(x,C_1\rho(N))\\[2mm]
&\subset&  \mathcal F_{x}^{ i_0(x)-1}(x,C_1\rho(N)).
\end{eqnarray*} Furthermore,  if $a_{E_{i_0(x)-1}}^-(x)\le-\frac{\log2}{2}$, then $[y]^{1\cdots (i_0(x)-1)}=\{x\}$, thus $$B_{\infty}(f,x,  C_1\rho(N))=\{x\}.$$ 
\end{proof}

\section{Tail entropy along leaves}

By Lemma \ref{local ball and foliation},  given $N\ge N_0$, $\mu\in \mathcal M_{inv}(f, \Lambda)$, without loss of generality, for $\mu$-a.e.,  $x$,  we may assume that $B_{\infty}(f,x, C_1\rho(N)) \subset \mathcal F_{x}^{ i}(x,C_1\rho(N))$ for some  $i$. Therefore, in what follows, we only need analyze the dynamics on leaves $\mathcal F_y^{*}(y)$,  $*\in \{1,\cdots, \ell\}$, $y\in\Lambda$. For the simplicity of symbols, we write $V_{y}^*=\mathcal F_y^{ *}(y)$. Moreover, we denote by $B_{V_y^*}(z, \rho)$ the ball in $V_y^*$ centered at $z$ with radius $\rho$, and define Bowen balls along leaves as follows
$$B_{V_y^*, n}(z,\rho)=\big{\{} p\in V_y^*: d_{V_{f^j(y)}^*}(f^j(p), f^j(z))<\rho,\, |j|<n\big{\}},$$
where $d_{V}$ denotes the distance in a submanifold $V\subset M$.  
For the convenience of computations, we intend to approximate the local complexity of dynamical systems   by that of their linearity.  Taking local trivializations,  we may assume $V_y^*\subset \mathbb{R}^{\dim E_*}$. Note that there exists a constant $C_2>0$ depending only  on  $\dim M$  such that 
 for any $1\le j\le \dim M$ and any linear map $X: \mathbb{R}^{j}\to \mathbb{R}^{j}$, one has $$\Gamma(X(B_{\mathbb R^{j}}(0,1)), \mathbb R^{j},   1/2)\le C_2\|X^{\wedge}\|^+,$$ 
 where $\Gamma(U, V,  \rho)$ denotes the minimal cardinality of covers for $U$ whose elements  are balls with radius $\rho$  in a manifold $V$, and $\|X^{\wedge}\|^+=e^{\log^+\|X^{\wedge}\|}$.

\begin{Lem}\label{scaling}
There exists $\eta_1>0$ such that for any $y\in \Lambda$, $*\in \{1,\cdots, \ell\}$,  $z\in B_{V_y^*}(y, C_1\rho(N))$ and  $\eta\in (0,\eta_1)$, 
$$\Gamma(f^{\pm N}(B_{V_y^*}(z,\eta)), V_{f^{\pm N}(y)}^*,  \eta/2))\le C_2e^{\frac{2}{N}} \|(D_yf^{\pm N}\mid_{E_{i}})^{\wedge}\|^+.$$

\end{Lem}

\begin{proof}
From the definition of $\rho(N)$, for $*\in \{1,\cdots, \ell\}$, $z\in B_{V_y^*}(y, C_1\rho(N))$,
$$\|(D_zf^{\pm N}\mid_{T_zV_y^*})^{\wedge}\|^+\le e^{\frac{1}{N}}\|(D_yf^{\pm N}\mid_{E_{i}})^{\wedge}\|^+.$$
For $\eta>0$,  define $g_{\eta, z}(p)=\eta p+z$, $z\in V_y^*$. Let $F_{\pm N, \eta, z}(p)=g_{\eta, f^{\pm N}(z)}^{-1}\circ f^{\pm N} \circ g_{\eta, z}(p)$. Then 
\begin{eqnarray*}
\|F_{\pm N, \eta, z}(p)-D_zf^{\pm N}\mid_{T_zV_y^*}(p)\|\,\,\text{converges to}\,\, 0, \quad \text{as}\,\,\eta\to 0,
\end{eqnarray*}
uniformly for $p\in B_{\mathbb R^{\dim E_*}}(0, 1)$, $z\in  B_{V_y^*}(y, C_1\rho(N))$. 
Observe that $$\Gamma(f^{\pm N}(B_{V_y^*}(z,\eta)),  V_{f^{\pm N}(y)}^*,  \eta/2)=\Gamma(F_{\pm N, \eta, y}(B_{\mathbb R^{\dim E_*}}(y,1)),  \mathbb R^{\dim E_*},  1/2).$$ So, there exists $\eta_1>0$ uniformly such that for any $\eta\in (0,\eta_1)$,  
\begin{eqnarray*}
&&\Gamma(f^{\pm N}(B_{ V_{y}^*}(z,\eta)), V_{f^{\pm N}(y)}^*,  \eta/2)\\[2mm]
&\le& e^{\frac{1}{N}}\Gamma((D_zf^{\pm N}\mid_{T_zV_y^*}(B_{\mathbb R^{\dim E_*}}(0,1)), \mathbb R^{\dim E_*}, 1/2)\\[2mm]
&\le& C_2\, e^{\frac{1}{N}} \|(D_zf^{\pm N}\mid_{T_zV_y^*})^{\wedge}\|^+\\[2mm]
&\le& C_2\, e^{\frac{2}{N}} \|(D_yf^{\pm N}\mid_{E_*})^{\wedge}\|^+.
\end{eqnarray*}

\end{proof}

Let $N\ge N_0$, $\mu\in \mathcal M_{inv}(f, \Lambda)$, then for $\mu$-a.e.,  $x$,  there exists  $i$  such that $B_{\infty}(f,x, C_1\rho(N)) \subset V_{x}^{ i}(x, C_1\rho(N))$. 
 For $\eta\in (0,\eta_1)$, let  $\{y_1,\cdots, y_{k(0)}\}$ be a finite $\eta$-net of  $ B_{V_x^i, n}(f^{N}, x, C_1\rho(N))$. Let $R_{j_0}=B_{V_x^i}(y_{j_0},\eta)\cap  B_{V_x^i, n}(f^{N}, x, C_1\rho(N)) $, $1\le j_0\le k(0)$. By induction, for $0\le s\le n-2$, suppose  $$y_{j_0,\cdots, y_{j_s}},\quad R_{j_0,\cdots, j_s}: \quad 1\le j_0\le k(0),\quad 1\le j_t\le k(0,j_0,\cdots, j_{t-1}),\,\,1\le t\le s,$$ have been defined.   Given  $y_{j_0,\cdots, y_{j_s}}$,  using Lemma \ref{scaling}, one may take a set $D$ which  is an $\eta/2$-net of $f^N(B_{V_{f^{sN}(x)}^i}(y_{j_0,\cdots, y_{j_s}},\eta))$ and has cardinality not more than  $C_2e^{\frac{2}{N}} \|(D_{f^{sN}(x)}f^{N}\mid_{E_{i}})^{\wedge}\|^+$. Observe that from the $\eta/2$-net $D$,  we can choose a set  $$\{y_{j_0,\cdots, j_s, j_{s+1}}: 1\le j_{s+1}\le k(0,j_0,\cdots, j_{s})\}$$ with $k(0,j_0,\cdots, j_{s})\le \sharp D$, which forms an $\eta$-net of $f^N(B_{V_{f^{sN}(x)}^i}(y_{j_0,\cdots, y_{j_s}},\eta)) \cap f^{N}(R_{j_0,\cdots, y_{j_s}}) $.  For $1\le j_{s+1}\le k(0,j_0,\cdots, j_{s})$,  denote $$R_{j_0,\cdots, j_s, j_{s+1}}=B_{V_{f^{(s+1)N}(x)}^i}(y_{j_0,\cdots, j_s, j_{s+1}}, \eta)\cap f^{N}(R_{j_0,\cdots, j_s}). $$
 In this way we could define all situations for $0\le s\le n-1$. 
 
For $ 1\le j_0\le k(0)$, $1\le j_t\le k(0,j_0,\cdots, j_{t-1}),$ $1\le t\le n-1$,  define $$U_{j_0,\cdots, j_{n-1}}=\{y\in B_{V_x^i, n}(f^{N}, x, C_1\rho(N)): f^{tN}(y)\in R_{j_0,\cdots, j_t},\quad 0\le t\le n-1 \}.$$
then 
$$\bigcup_{j_0,\cdots, j_{n-1}} U_{j_0,\cdots, j_{n-1}}= B_{V_x^i, n}(f^{N}, x, C_1\rho(N)).$$
Note that for any $y, z\in U_{j_0,\cdots, j_{n-1}}$, $0\le t\le n-1$, 
\begin{eqnarray*}
&&d_{V_{f^{tN}(x)}^i}(f^{tN}(y), f^{tN}(z))\\[2mm]
&\le&d_{V_{f^{tN}(x)}^i}(f^{tN}(y), y_{j_0,\cdots, j_t})+d_{V_{f^{tN}(x)}^i}(f^{tN}(z), y_{j_0,\cdots, j_t})\le 2\eta. 
\end{eqnarray*} 
 Therefore, 
\begin{eqnarray*}
&& r_n(f^{N}, B_{V_x^i, n}(f^{ N}, x, C_1\rho(N)), 2\eta) \\[2mm]
 &\le&  \sum k(0,j_0,\cdots, j_{n-2}) \le  k(0)\cdot \Pi_{t=0}^{n-2}(C_2e^{\frac{2}{N}}\|(D_{f^{tN}(x)}f^{ N}\mid_{E_{i}})^{\wedge}\|^+),
\end{eqnarray*}
which implies
\begin{eqnarray*}
	&&\limsup_{n\rightarrow\infty}\frac{1}{n}\log r_n(f, B_{V_x^i, n}(f, x, C_1\rho(N)), 2\eta)\\[2mm] & \le&  \limsup_{n\rightarrow\infty}\frac{1}{nN}\log r_n(f^N, B_{V_x^i, n}(f^N, x, C_1\rho(N)), 2\eta )\\[2mm]
	&\le& \limsup_{n\rightarrow\infty}\frac{1}{nN}\log( \Pi_{t=0}^{n-2}(C_2e^{\frac{2}{N}}\|(D_{f^{tN}(x)}f^{ N}\mid_{E_{i}})^{\wedge}\|^+))\\[2mm]
	&\le& \frac{2+\log C_2}{N}+ \Delta_f^+( x, E_i, N).
\end{eqnarray*}
By the arbitrariness of $\eta$, we  obtain 
\begin{eqnarray*}
h^*(f,x, \rho(N))&\le& \lim_{\eta\to0} \limsup_{n\rightarrow\infty}\frac{1}{n}\log r_n(f, B_{V_x^i, n}(f, x, C_1\rho(N)), 2\eta)\\[2mm]
&\le&  \frac{2+\log C_2}{N}+ \Delta_f^+( x, E_i, N).
\end{eqnarray*}
Similarly, considering the inverse $f^{-1}$, we get 
\begin{eqnarray*}
	h^*(f^{-1},x, \rho(N))&\le& \lim_{\eta\to0} \limsup_{n\rightarrow\infty}\frac{1}{n}\log r_n(f^{-1}, B_{V_x^i, n}(f^{-1}, x, C_1\rho(N)), 2\eta)\\[2mm]
	&\le&  \frac{2+\log C_2}{N}+ \Delta_f^-(x, E_i, N).
\end{eqnarray*}

\section{Measure theoretic tail entropy and upper semi-continuity}

In this section, we first analyze the relationship between the scale of measure theoretic tail entropy and the evolution time, and hence give the proof of Theorem \ref{uppersemi}. 

\begin{proof}[Proof of Theorem \ref{uppersemi}] For $N\ge N_0$,  let $\vep_N=\rho(N)$. If $\mu\in \mathcal M_{erg}(f, \Lambda)$,  then $h^*(f^{\pm},x, \vep_N)$  are constants   for $\mu$-a.e., $x$, which we denote by $h^*(f^{\pm1}, \mu, \vep_N)$. 
	By Proposition 2.7 of \cite{CY}, one further obtains  $$h^*(f, \mu, \vep_N)=h^*(f^{-1},  \mu, \vep_N).$$
	Hence,  
	\begin{eqnarray*}
		h^*(f,\mu, \vep_N)&\le &   \frac{2+\log C_2}{N}+ \min\{ \Delta_f^{\pm}(\mu, E_i;  N): 1\le i\le \ell\} \\[2mm]
		&=&\frac{2+\log C_2}{N}+  \Delta_f(\mu, N).
	\end{eqnarray*}	
	When $\mu\in\mathcal M_{inv}(f, \Lambda) $, 
	using ergodic decomposition $\mu=\int_{\mathcal M_{erg}(f, \Lambda)}d\tau(m)$, 	we deduce 
	\begin{eqnarray*}
		h^*(f, \mu, \vep_N) 
		&\le&  \frac{2+\log C_2}{N}+\int_{\mathcal M_{erg}(f, M)} \Delta_f(m,  N)d\tau(m)\\[2mm]
		&=& \frac{2+\log C_2}{N}+\ \Delta_f(\mu, N),
	\end{eqnarray*}	
	which gives rise to  
	$$ \sup_{\mu\in \mathcal{M}_{inv}(f, \Lambda)} \Big{(}h^*(f,\mu, \vep_N)-\Delta_f(\mu, N)\Big{)} \leq \frac{2+\log C_2}{N} \to  0,\quad \text{as}\,\,N\to +\infty.$$ 
	In particular, since $\Delta_f(\mu, N) \to \Delta_f(\mu) $ as $N\to +\infty$, it holds that 
	$$\lim_{\vep\to 0} h^*(f,\mu, \vep) \leq \Delta_f(\mu)$$
	for any $\mu\in \mathcal{M}_{inv}(f, \Lambda)$.
\end{proof}

Next we are going to prove Theorem \ref{tail of metric entropy}.
\begin{proof}[Proof of Theorem \ref{tail of metric entropy}] By Jacobs theorem (see Theorem 8.4 of \cite{Walters}), it suffices to consider $\mu$ to be ergodic. Moreover,  by Proposition 2.1 of \cite{LSW}(note that the finity of topological entropy is used in the proof there), it is in fact enough to  prove that for $\mu$-a.e., $x$, 	
\begin{eqnarray*}
&&\lim_{\delta\to 0}\limsup_{n\rightarrow0}\frac{1}{n}\log r_n(f, B_{n}(f, x, \rho), \delta)\\[2mm]
&\le& \lim_{\delta\to 0}\limsup_{n\rightarrow0}\frac{1}{n}\log r_n(f, B_{\infty}(f, x,\rho), \delta)= h^*(f,\mu,\rho).
\end{eqnarray*}	
Note that, given $\gamma>0$,  for $\mu$-a.e., $x$, there exist $L(x)\in \mathbb{N}$ and a finite subset $D_{L(x)}(x)\subset B_{\infty}(f, x, \rho) $ with $\cup_{y\in D_{L(x)}(x)} B_{L(x)}(f, y,\delta)\supset B_{\infty}(f, x, \rho) $ satisfying 
$$\sharp D_{L(x)}(x)=r_{L(x)}(f, B_{\infty}(f, x, \rho), \delta)\le e^{L(x)(h^*(f,\mu,\rho)+\gamma)}.$$
Furthermore,  one may choose $T(x)\in \mathbb{N}$ such that 
$$ \bigcup_{y\in D_{L(x)}(x)} B_{L(x)}(f, y,\delta)\supset B_{T(x)}(f, x, \rho), $$ 
which implies 
$$r_{L(x)}(f, B_{T(x)}(f, x, \rho), \delta)\le \sharp D_{L(x)}(x) \le  e^{L(x)(h^*(f,\mu,\rho)+\gamma)}. $$
For any $j\in \mathbb{N}$, denote  
$Y_j=\{x: L(x)\le j,\, T(x)\le j\}$, then $\mu(Y_j)\to 1$ as $j\to +\infty$. For $\mu$-a.e., $x$,  by the ergodicity of $\mu$,  for large $n$,  one has 
$$\frac{\sharp  \{0\le k< n: f^k(x)\notin Y_j\}}{n} \le  1-\mu(Y_j)+\frac{1}{j}.$$
We define a sequence $0=n_0<n_1<\cdots<n_{k-1}<n_k=n$ of integers by induction. Suppose $n_s$ is defined, then  \begin{equation*} \begin{cases} 
n_{s+1}=n_s+L(f^{n_s}(x)), \text{ if}\, f^{n_s}(x)\in Y_j\,\text{and}\, n_s+j\le n;\\  
n_{s+1}=\min\{t>n_s: f^t(x)\in Y_j\}, \text{if}\,f^{n_s}(x)\notin R_j\text{ and}\, \min\{t>n_s: f^t(x)\in Y_j\}\le n;\\  n_{s+1}=n,\,\,\text{ otherwise}. \end{cases} \end{equation*}  
Since the elements of $\{x, f(x),\cdots, f^{n-1}(x)\}$ outside $Y_j$ don't exceed $n(1-\mu(Y_j)+\frac{1}{j})$, by  Lemma 2.1 of \cite{Bowen}, 
 \begin{eqnarray*}r_n(B_n(f, x, \rho), 2\delta)
 \le    e^{n(h^*(f,\mu,\rho)+\gamma)}\cdot r_1(f, M,\delta)^{n(1-\mu(Y_j)+\frac{1}{j})+j},
\end{eqnarray*}
which implies
\begin{eqnarray*}
&&\limsup_{n\rightarrow +\infty}\frac{1}{n}\log r_{n}(f, B_{n}(f, x, \rho), 2\delta)\\[2mm]
 &\le&  h^*(f,\mu,\rho)+\gamma+(1-\mu(Y_j)+\frac{1}{j})\log r_1(f, M,\delta).
\end{eqnarray*}
Since $j$ and $\gamma$ are arbitrary, it follows that 
\begin{eqnarray*}
	\limsup_{n\rightarrow +\infty}\frac{1}{n}\log r_{n}(f, B_{n}(f, x, \rho), 2\delta)\le h^*(f,\mu,\rho).
\end{eqnarray*}
Letting $\delta\to0$,  we finish the proof of Theorem  \ref{tail of metric entropy}. 
\end{proof}

Now together with the uniform arguments in Theorem \ref{uppersemi} and Theorem \ref{tail of metric entropy}, we are in a position to prove  Corollary \ref{uppersemicon}.

\begin{proof}[Proof of Corollary \ref{uppersemicon}] If $\Delta_f(\mu)=0$, then given $\delta>0$, for  large $N\in \mathbb{N}$ one has  $\Delta_f(\mu, N)\le \delta$. Besides, by  Theorem \ref{uppersemi},   taking $N$ sufficiently large in advance,  it holds that $$h^*(f,\nu, \vep_N)\le \Delta_f(\nu, N)+\delta$$
for any $\nu\in \mathcal M_{inv}(f, \Lambda)$.	
	 Note that  $\Delta_f(\nu, N)$  is continuous relative to  $\nu\in \mathcal M_{inv}(f, \Lambda)$, so for $\nu$ close to $\mu$,   we have $\Delta_f(\nu, N)\le 2\delta$  and hence
 $$h^*(f,\nu, \vep_N)\le 3\delta.$$
 Let $\mathcal P$ be a finite measurable partition with  $\mu(\partial(\mathcal P))=0$ and  $diam(\mathcal P) \leq \vep_N.$ By Theorem  \ref{tail of metric entropy}, 
\begin{eqnarray*}
h_{\nu}(f)-h_{\nu}(f, \mathcal P)\leq h^*(f,\nu, \vep_N)\le 3\delta.
\end{eqnarray*}
Moreover, for the fixed $\mathcal P$, $h_{\nu}(f, \mathcal P)$ is upper semi-continuous  at $\mu$, which implies  
\begin{eqnarray*}
h_{\nu}(f, \mathcal P)\le h_{\mu}(f, \mathcal P)+\delta,
\end{eqnarray*}
when $\nu$ close to $\mu$.  Therefore, 
\begin{eqnarray*}
	h_{\nu}(f)\le h_{\nu}(f, \mathcal P)+3\delta\le h_{\mu}(f, \mathcal P)+4\delta,
\end{eqnarray*}
which  consequently, combining with the arbitrariness of $\delta$, gives the upper semi-continuity of metric entropy at $\mu$ in  $ \mathcal M_{inv}(f, \Lambda)$.  The proof of Corollary \ref{uppersemicon} is completed. 	
\end{proof}

\end{document}